\documentclass[12pt,centertags,oneside]{article}
\usepackage{amsmath,amstext,amsthm,amscd,typearea}
\usepackage{amssymb}
\usepackage{a4wide}
\usepackage[mathscr]{eucal}
\usepackage{mathrsfs}
\usepackage{typearea}
\usepackage{charter}
\usepackage{pdfsync}
\usepackage{url}

\usepackage{pst-node}
\usepackage{tikz-cd}

\usepackage{xcolor}

\usepackage[a4paper,width=15.2cm,top=3cm,bottom=3cm]{geometry}

\numberwithin{equation}{section}

\newtheorem{theorem}{Theorem}[section]

\newtheorem{proposition}[theorem]{Proposition}
\newtheorem{corollary}[theorem]{Corollary}
\newtheorem{lemma}[theorem]{Lemma}
\newtheorem{remark}[theorem]{Remark}

\newcommand{\cali}[1]{\mathscr{#1}}

\newcommand{\vol}{\mathop{\mathrm{vol}}}

\newcommand{\ddc}{dd^c}

\newcommand{\PSH}{{\rm PSH}}

\newcommand{\B}{\mathbb{B}}
\newcommand{\C}{\mathbb{C}}

\newcommand{\N}{\mathbb{N}}

\newcommand{\R}{\mathbb{R}}
\renewcommand\P{\mathbb{P}}

\title{\bf  Higher Lelong numbers versus full Monge-Amp\`ere mass}
\providecommand{\keywords}[1]{\textbf{\textit{Keywords:}} #1}
\providecommand{\subject}[1]{\textbf{\textit{Mathematics Subject Classification 2010:}} #1}

\author{Do Duc Thai and Duc-Viet Vu}
\newcommand{\Addresses}{{
		\bigskip
		\footnotesize
\noindent
		\textsc{Duc-Viet Vu, University of Cologne, Division of Mathematics, Department of Mathematics and Computer Science, Weyertal 86-90, 50931, K\"oln,  Germany.}
		\noindent
		\par\nopagebreak
		\noindent
		\textit{E-mail address}: \texttt{vuviet@math.uni-koeln.de}
		\newline
	
		\noindent
		\textsc{Do Duc Thai, Department of Mathematics, Hanoi National University of Education, 136 XuanThuy str., Hanoi, Vietnam.}
		\noindent
		\par\nopagebreak
		\noindent
		\textit{E-mail address}: \texttt{doducthai@hnue.edu.vn}	
}}

\date{\today}
\begin{document}
\maketitle
\begin{abstract} We prove a recent conjecture of Chi Li relating the notion of higher Lelong numbers to that of full Monge-Amp\`ere mass. 
\end{abstract}
\noindent
, {full  Monge-Amp\`ere mass}, {non-pluripolar product}.
\\

\noindent
\subject{32U15},  {32Q15}.

%\tableofcontents

%%%%%%%%%%%%%%%%%%%%%%%%%%%%%
%%%%%%%%%%%%%%%%%%%%%%%%%%%%%%%%%%

\section{Introduction}

Let $n\ge 2$ be an integer. Let $z$ be the standard coordinates on $\C^n$.  Denote by $\|z\|$ the Euclidean norm of $z$. Let $\psi$ be a plurisubharmonic function (psh for short) on an open neighborhood $U$ around $0$ in $\C^n$ such that $\psi$ is locally bounded outside $0$. Recall that $d^c:= \frac{i}{2\pi} (\bar \partial - \partial)$, and $\ddc = \frac{i}{\pi} \partial \bar \partial$.  For $1 \le k \le n$, \emph{the higher Lelong number} of $\psi$ of order $k$ is given by 
$$e_k(\psi):= \int_{\{0\}} (\ddc \psi)^k \wedge (\ddc \log \|z\|)^{n-k}.$$
The number $e_1(\psi)$ is the Lelong number of $\psi$ at $0$ and  is usually denoted by $\nu(\psi,0)$. The higher Lelong numbers are closely related to the notion of log canonical threshold. We refer to \cite{Demailly-Hiep} for a fundamental lower bound of the log canonical threshold in terms of higher Lelong numbers, see also \cite{deFernex-Ein-Mustata,deFernex-Ein-Mustata2,Pukhlikov} for applications to birational geometry.  If $g_1, \ldots, g_N$ are holomorphic functions defined on an open neighborhood of $0$ in $\C^n$ and $\psi=\log \sum_{j=1}^N |g_j|$, then $e_n(\psi)$ is equal to the Samuel multiplicity of the ideal generated by germs of $g_1, \ldots, g_N$ at $0$; see \cite{Demailly_localalgebra}. These facts are instances illustrating the relevance of higher Lelong numbers to algebraic geometry. For more information about those numbers, we refer to \cite{Demailly-Hiep,Kim-Rashkovskii,Kim-Rassh-assymptotic} and references therein.

A major problem in the theory of singularity of psh functions is to clarify the relation between higher Lelong numbers and the ideal multiplier sheaves associated to $\psi$. In this research direction, Demailly conjectured that for every psh function $\psi$ locally bounded outside $0$, the sequence of higher Lelong numbers at $0$ of psh functions with analytic singularities approximating $\psi$ in his analytic regularization of psh functions would converge to those of $\psi$. If this were true, it would imply the well-known conjecture by Guedj and Rashkovskii independently  (see \cite{DGZ,Rashkovskii-open}) stating that $e_n(\psi)=0$ if and only if $e_1(\psi)=0$ (recall however that by \cite{Demailly-Hiep}, one has that $e_k(\psi)=0$ if and only if $e_1(\psi)=0$ for every $1 \le k \le n-1$). The latter conjecture has remained largely open although some important cases were confirmed such as when $\psi$ is toric (\cite{Kim-Rashkovskii}) or has sup-analytic singularities (\cite{Boucksom-Favre-Jonsson,Kim-Rassh-assymptotic,R-approx-psh}). We refer to  \cite{Kim-Rassh-assymptotic} for a detailed summary of developments regarding the  Guedj-Rashkovskii conjecture.  

Recently, Chi Li \cite{ChiLi} constructed a counter-example to the above conjecture of Demailly. The importance of such a counter-example lies in the fact that it implies that in general the higher Lelong numbers can not be recovered using the usual analytic approximation from above as in  Demailly's analytic regularisation theorem. Nevertheless since the example in \cite{ChiLi} only concerns with psh functions of positive Lelong number at $0$, it gives us no clear clue how to deal with the Guedj-Rashkovskii conjecture. In the same paper Chi Li posed a conjecture extending his construction to a much larger natural class of psh functions (which still have positive Lelong number at $0$).  The aim of this paper is to prove this conjecture. 

To go into details, we need a few more things.  We are interested in $\psi$ such that $e_1(\psi)\ge 1$, or in other words, $\psi \le \log \|z\|+ O(1)$ near $0$. In this case, by monotonicity (see  \cite[Lemma 4.1]{AhagCegrellHiepMA} or \cite{Demailly_ag}), one has 
\begin{align}\label{ine-sosanhekj}
1 \le e_1(\psi) \le \cdots  \le e_n(\psi).
\end{align}
 Let $\P^{n-1}$ be the complex $(n-1)$-dimensional projective space, and  $\pi:\C^n \backslash \{0\} \to \P^{n-1}$ be the natural projection. Let $\omega_{n-1}$ be the Fubini-Study form on $\P^{n-1}$, and let $u$ be an $\omega_{n-1}$-psh function. Hence $\ddc u + \omega_{n-1} \ge 0$.  Since $\pi^* \omega_{n-1}= \ddc \log \|z\|$, we see that $\pi^* u+ \log \|z\|$ is a psh function on $\C^{n} \backslash \{0\}$   which is bounded from above near $0$. Thus it can be extended to a psh function on $\C^n$.  Put 
$$\varphi:= \max \{ \pi^* u+\log \|z\|, 2 \log \|z \|\}.$$

In \cite{ChiLi}, it was predicted that $e_n(\varphi)=1$ if and only if $u$ is of full Monge-Amp\`ere mass, \emph{i.e.,} 
$$\int_{\P^{n-1}}\langle (\ddc u+\omega_{n-1})^{n-1}\rangle = 1,$$ 
where for every closed positive $(1,1)$-current $T$, the expression $\langle T^n \rangle$ denotes the non-pluripolar product of $T, \ldots, T$ ($n$ times). The conjecture was verified in \cite{ChiLi} when $u$ has analytic singularities, or $e^u$ is continuous and $\{u=-\infty\}$ contained in a compact subset of an affine chart of $\P^{n-1}$.   As already observed in \cite{ChiLi}, this conjecture can be seen as a local version of a characterization of quasi-psh functions of full Monge-Amp\`ere mass obtained in \cite{Darvas-kahlerclass,Lu-Darvas-DiNezza-mono}. We refer to \cite{BT_fine_87,BEGZ,GZ-weighted,Viet-generalized-nonpluri} for information about the notion of non-pluripolar products. Let 
$$\tilde{e}_k(u):= 1- \int_{\P^{n-1}} \langle (\ddc u+\omega_{n-1})^k \rangle \wedge \omega_{n-1}^{n-1-k}.$$
These numbers satisfy
\begin{align}\label{ine-ehigher}
0 \le \tilde{e}_1(u) \le \cdots \le \tilde{e}_{n-1}(u) \le 1
\end{align}
%\int_{\P^{n-1}}\langle (\ddc u+\omega)^{n-1}\rangle \le  \int_{\P^{n-1}} (\ddc u+\omega_{n-1})^k \wedge \omega_{n-1}^{n-1-k} \le  1,$$ 
by the monotonicity of non-pluripolar products;  see (\ref{ine-moniPn}) below. If $u$ is bounded, then $\tilde{e}_k=0$ for every $1 \le k \le n-1$. Hence, one can see that $(\tilde{e}_k)_k$ encodes information about the singularity of $u$. Indeed it was proved that $\tilde{e}_k(u)>0$ if $u$ has a positive generic Lelong number along an analytic set of dimension at least $n-1-k$. For more precise and quantitative estimates, we refer to \cite{Vu_lelong-bigclass,Vu_lelong-bignef-quantitative} and references therein. The following theorem is our main result confirming in particular the above conjecture of Chi Li. 

\begin{theorem} \label{the-chili-lelong}  There is a constant $C>0$ independent of $u$ such that
$$\tilde{e}_{k-1}(u) \le e_k(\varphi) -1 \le  C \tilde{e}_{k'}(u)$$
for every $1 \le k \le n$, where $k':= \min \{k, n-1\}$.  In particular, $e_n(\varphi)=1$ if and only if $u$ is of full Monge-Amp\`ere mass.
\end{theorem} 

For every psh function $\psi$ on an open subset $U$ of $\C^n$, we recall that  the multiplier ideal sheaf $\cali{I}(\psi)$ is the ideal sheaf whose stalk $\cali{I}_z(\C)$ at $z \in U$ consists of holomorphic germs $f$ at $z$ such that $|f|^2 e^{-2\psi}$ is locally integrable with respect to the Lebesgue measure on $\C^n$ near $z$.
 For psh functions $\psi_1, \psi_2$ defined on an open neighborhood of $0$ in $\C^n$, we recall that $\psi_1$ and $\psi_2$ are \emph{valuatively equivalent at $0$} if 
$$\cali{I}_0(m \psi_1)= \cali{I}_0(m \psi_2)$$
for every constant $m>0$. By  \cite[Corollary 10.18]{Boucksom_l2} (or \cite[Theorem A]{Boucksom-Favre-Jonsson}), it is well-known that  $\psi_1$ are $\psi_2$ are valuatively equivalent at $0$ if and only if for every local smooth modification $\rho: X \to U$ ($U$ is an open neighborhood of $0$) and every irreducible hypersurface $E \subset \rho^{-1}(0)$, the generic Lelong numbers of $\psi_1\circ \rho$ and $\psi_2\circ \rho$ along $E$ are equal.

% Consider such a function $u$ on $\P^{n-1}$.  Hence one gets $\tilde{e}_1(u)>0$,   and  $\tilde{e}_{k}(u)>0$ for every $1 \le k \le n-1$ by (\ref{ine-ehigher}). Thus, in this case,  the above result shows that for $2 \le k \le n$, one has  $e_k(\varphi)>1$.

Recall that a psh germ $\psi$ at $0$ (\emph{i.e.,} an equivalence class of psh functions on open neighborhoods of $0$ which are identical on some small open neighborhood of $0$) is said to have \emph{analytic singulatities} if there is a small enough open neighborhood $U$ of $0$ such that 
$$\psi - c \log (|f_1|+ \cdots+ |f_M|)$$
is a bounded function on $U$ for some holomorphic functions $f_1,\ldots, f_M$ on $U$, and for some constant $c>0$. For every psh germ $\psi$ at $0$, let $\cali{A}_\psi$ be the set of psh germs $\psi'$ with analytic singularities such that $\psi'$ is less singular than $\psi$, that means  $\psi \le \psi'+C$ near $0$ for some constant $C$.

The following consequence of Theorem \ref{the-chili-lelong} tells us that  in general one can not recover higher Lelong numbers of a given psh germ $\psi$ by using those with analytic singularities approximating it from above.   

\begin{corollary} \label{cor-phanvidu} Assume that $u$ has zero Lelong number everywhere and $\ddc u+ \omega_{n-1}$ has mass on some pluripolar set in $\P^{n-1}$. Then we have
$$\sup_{\varphi' \in \cali{A}_\varphi} e_k(\varphi')=1 < e_k(\varphi), \quad \text{ for }\, 2 \le k \le n.$$
\end{corollary}

For a more general version of this corollary, we refer to Theorem \ref{the-chili-lelong2v} and the comment following it. It is well-known that there are examples of quasi-psh functions on $\P^1$ having zero Lelong numbers everywhere but their Laplacian charges some pluripolar set in $\P^1$. We will explain in  Section \ref{sec-zerolelong} how to obtain from this example  plenty of closed positive $(1,1)$-currents on projective manifolds such that they have zero Lelong number everywhere and  charge some pluripolar sets. Here we say that a closed positive current $T$ charges a set $A$ if the trace measure of $T$ has positive mass on $A$.

Note that  since $\varphi$ and $\log \|z\|$ are valuatively equivalent at $0$ (see Lemma \ref{le-idealsheaf} below or \cite[Corollary 2.2]{ChiLi}),  Corollary \ref{cor-phanvidu} shows that even if two psh germs are valuatively equivalent (at $0$), then their Monge-Amp\`ere masses at $0$ are not necessarily equal in general. We refer to \cite{Kim-Rassh-assymptotic} for related discussions.  %We expect however that the converse is true under suitable contexts. We refer to Problem \ref{que-valuativepresribe} and the paragraphs following it in Section \ref{sec-proofmain1}for more details.  

We would like to have some comments on the proof of Theorem \ref{the-chili-lelong}. Unlike the setting in \cite{ChiLi}, in our present situation (which is of much greater generality) one can not use directly the classical intersection of $(1,1)$-currents (in the sense given in \cite{Bedford_Taylor_82,Demailly_ag}). We will use the notion of non-pluripolar products introduced in \cite{BT_fine_87,BEGZ,GZ-weighted} (and also \cite{Viet-generalized-nonpluri} for generalizations) as a substitute to the classical intersection of currents. However this substitution is, as one can guess, not satisfactory.  A crucial difficulty when dealing with the non-pluripolar products is the lack of proper calculus for them (continuity and integration by parts). Such things were known under certain conditions which are not available in our situation. Thus, although we will still use a formula due to B\l ocki \cite{blocki-ddcmax} to compute the Monge-Amp\`ere operators of maxima of psh functions as in \cite{ChiLi}, we will need a new argument afterward to estimate these Monge-Amp\`ere currents. Our key ingredient is the monotonicity of non-pluripolar products proved in \cite{BEGZ,Lu-Darvas-DiNezza-mono,Viet-generalized-nonpluri,WittNystrom-mono}.

The paper is organized as follows. The main result is proved in Section \ref{sec-proofmain1}. We give examples of quasi-psh functions with zero Lelong number in Section \ref{sec-zerolelong}. 
\\

\noindent
\textbf{Acknowledgement.} We would like to thank referees for numerous suggestions improving greatly the presentation of the paper. We are also grateful to  Nguyen Ngoc Cuong and Dano Kim for fruitful discussions.  The research of Do Duc Thai is supported by an NAFOSTED grant of Vietnam (Grant No. 101.04-2021.31).  The research of D.-V. Vu is partially funded by the Deutsche Forschungsgemeinschaft (DFG, German Research Foundation)-Projektnummer 500055552 and by the ANR-DFG grant QuaSiDy, grant no ANR-21-CE40-0016.  %D.-V. Vu  is supported by a postdoctoral fellowship of the Alexander von Humboldt Foundation. 

\section{Proof of Theorem \ref{the-chili-lelong}} \label{sec-proofmain1}

We first recall some basic properties of non-pluripolar products from \cite{BEGZ}. Let $\Omega$ be an open subset in $\C^n$. Let $u_1,\ldots, u_m$ be psh functions on $\Omega$. For every Borel set $A$, we denote by $\bold{1}_A$ the characteristic function of $A$, \emph{i.e.}, it is equal to $1$ on $A$ and $0$ outside $A$. 

 For $k \ge 0$, put $u_{jk}:= \max\{u_j, -k\}$ for $1 \le j \le m$ and $R_k:= \ddc u_{1k} \wedge \cdots \wedge \ddc u_{mk}$. By plurifine locality of Monge-Amp\`ere operators (\cite{BT_fine_87}), one has 
$$\bold{1}_{\cap_{j=1}^m \{u_j >-k\}}R_k= \bold{1}_{\cap_{j=1}^m \{u_j >-k\}}R_l$$
for every $l \ge k$. Hence we obtain an increasing family of positive currents $R'_k:=\bold{1}_{\cap_{j=1}^m \{u_j >-k\}}R_k$ on $\Omega$. If $R'_k$ is of mass bounded uniformly on compact sets in $\Omega$, then the weak limit $\lim_{k \to \infty} R'_k$ exists and we denote it by $\langle \ddc u_1\wedge \cdots \wedge\ddc u_m \rangle$. The last current (if exists) is called \emph{the non-pluripolar product of $\ddc u_1,\ldots, \ddc u_m$}. It is a positive current on $\Omega$. If  the current $\langle \ddc u_1\wedge \cdots \wedge\ddc u_m \rangle$ is well-defined, then it is symmetric and multi-linear in $\ddc u_j$, and has no mass on pluripolar sets (see \cite[Proposition 1.4]{BEGZ}).  We now turn to the compact setting. 

Let $X$ be a compact K\"ahler manifold of dimension $n$. Let $T_1,\ldots, T_m$ be closed positive $(1,1)$-currents on $X$. We write $T_j= \ddc u_j$ locally for some psh function $u_j$. It was shown in \cite{BEGZ} that the non-pluripolar product of $\ddc u_1,\ldots,\ddc u_m$ is well-defined, and  independent of the choices of potentials. Thus we can glue it together to obtain a global current on $X$ which we call the non-pluripolar product of $T_1,\ldots, T_m$ and  denote  by $\langle T_1 \wedge \cdots \wedge T_m \rangle$. As in the local setting,  the non-pluripolar product $\langle T_1 \wedge \cdots \wedge T_m \rangle$ is symmetric and multi-linear in $T_1,\ldots, T_m$. Moreover the non-pluripolar product $\langle T_1 \wedge \cdots \wedge T_m \rangle$ is indeed a closed positive current (\cite[Theorem 1.8]{BEGZ}).  It is also a direct consequence of the construction of non-pluripolar products that if $T_1:= [D]$ the current of integration along a divisor $D$ in $X$, then 
\begin{align}\label{eq-bang0nonpluripolarD}
\langle [D] \wedge T_2 \wedge \cdots \wedge T_m \rangle =0
\end{align}
for every closed positive $(1,1)$-current $T_2,\ldots, T_m$. The reason is that the product $\langle [D] \wedge T_2 \wedge \cdots \wedge T_m \rangle$ has no mass on $D$ (which is a pluripolar set), and outside $D$, the current $[D]$ is identically zero.

For every closed positive current $S$, we denote by $\{S\}$ the cohomology class of $S$. For an integer $1\le p \le n$, and every cohomology $(p,p)$-classes $\alpha$ and $\beta$, we write $\alpha \le \beta$ if $\beta- \alpha$ is the cohomology class of a closed positive $(p,p)$-current.  Here is a crucial monotonicity property of  non-pluripolar products proved in \cite[Theorem 1.1]{Viet-generalized-nonpluri}.

\begin{theorem} \label{th-mono-current11} Let $1 \le m \le n $  be an integer.  Let $T_1,\ldots, T_m$ be closed positive $(1,1)$-currents on $X$.  Let $T'_j$ be a closed positive $(1,1)$-current in the cohomology class of $T_j$ such that $T'_j$ is less singular than $T_j$ (this means that $u_j \le u'_j+ O(1)$ if $u_j, u'_j$ are potentials of $T_j,T'_j$ respectively) for $1 \le j \le m$. Then we have 
\begin{align}\label{ine-monomnonpluri}
\{\langle T_1 \wedge \cdots \wedge T_m \rangle \}  \rangle \le \{ \langle T'_1 \wedge \cdots \wedge T'_m \rangle \}.
\end{align}
\end{theorem}

%Here we denote by $\langle T_1 \wedge \cdots \wedge T_m \rangle$ the non-pluripolar product of $T_1,\ldots, T_m$. 
A direct consequence of  (\ref{ine-monomnonpluri}) is the following inequality 
\begin{align}\label{ine-moniPn}
\int_{X} \langle \wedge_{j=1}^m T_j \rangle \wedge \omega^{n-m} \le  \int_{X} \langle \wedge_{j=1}^{m} T'_j \rangle \wedge \omega^{n-m},
\end{align}
which  was already proved in  \cite{BEGZ,Lu-Darvas-DiNezza-mono,WittNystrom-mono}, where $\omega$ is a K\"ahler form on $X$. The last inequality for $m=n$ was used often in the study of complex Monge-Amp\`ere equations. In our proof of the main result later, we need to make use of the full generality of Theorem \ref{th-mono-current11} for every $1 \le m \le n$ and $X= \P^{n-1}$. We also note that since the $p^{th}$ cohomology group $H^{p,p}(\P^{n-1},\R)$ of $\P^{n-1}$ is generated by  the class of  $\omega_{n-1}^p$ (which is the product of smooth closed positive $(1,1)$-forms), one can deduce (\ref{ine-monomnonpluri}) for $X= \P^{n-1}$ from (\ref{ine-moniPn}). 

The next formula is also one of the keys we need to use later.

\begin{lemma} \label{le-ddmax-blocki} Let $u$ and $v$ be psh functions on an open subset $U$ of $\C^n$. Let $2 \le m \le n$ be an integer. Then we have 
\begin{multline*}
\big\langle  (\ddc \max\{u,v\})^m \big \rangle =    \sum_{j=0}^{m-1} \big \langle \ddc \max\{u,v\} \wedge (\ddc u)^j \wedge (\ddc v)^{m-1-j}\big \rangle -\\
 \sum_{j=1}^{m-1}\big \langle  (\ddc u)^j \wedge (\ddc v)^{m-j}\big \rangle
\end{multline*}
\end{lemma}

\proof When $u$ and $v$ are bounded, the above formula was proved in \cite[Theorem 4]{blocki-ddcmax}. Hence for $u_k:= \max\{u,-k\}$  and $v_k:= \max\{v,-k\}$ (where $k>0$ is a constant), one gets
\begin{multline} \label{eq-uvblockibounded}
 (\ddc \max\{u_k,v_k\})^m =    \sum_{j=0}^{m-1}  \ddc \max\{u_k,v_k\} \wedge (\ddc u_k)^j \wedge (\ddc v_k)^{m-1-j} -\\
 \sum_{j=1}^{m-1} (\ddc u_k)^j \wedge (\ddc v_k)^{m-j}.
\end{multline}
Let $A_k:= \{u>-k\} \cap \{v>-k\}$. Since 
$$\max\{u_k,v_k\}= \max\{u,v\}= \max\{\max\{u,v\},-k\}=:w_k$$
 on $A_k$, by \cite{BT_fine_87}, we have $\bold{1}_{A_k}\ddc \max\{u_k, v_k\}=\bold{1}_{A_k}\ddc w_k$. This combined with (\ref{eq-uvblockibounded}) gives 
\begin{multline} \label{eq-uvblockibounded2}
 \bold{1}_{A_k}(\ddc w_k)^m =    \sum_{j=0}^{m-1} \bold{1}_{A_k} \ddc w_k \wedge (\ddc u_k)^j \wedge (\ddc v_k)^{m-1-j} -\\
 \sum_{j=1}^{m-1} \bold{1}_{A_k} (\ddc u_k)^j \wedge (\ddc v_k)^{m-j}.
\end{multline} 
Observe that $A_k = \{u>-k\} \cap \{v>-k\} \cap \{\max\{u,v\}>-k\}$. Thus the desired equality follows by letting $k \to 0$ in (\ref{eq-uvblockibounded2}) and using the definition of non-pluripolar products.
\endproof

Let $p:X \to Y$ be a holomorphic submersion between complex manifolds. Let $\Phi$ be a smooth form of degree $s$ with compact support in $X$. Recall that the push-forward $p_* \Phi$ is a smooth form of degree $s-2(\dim X-\dim Y)$ on $Y$ defined by integrating $\Phi$ along fibers of $p$; see \cite[Page 17]{Demailly_ag} for a detailed presentation. If $\Phi$ is closed, then so is $p_* \Phi$. 

Let $S$ be now current of degree $s$ on $Y$. The pull-back $p^* S$ of $S$ is defined as follows. For every smooth form $\Phi$ of degree $2\dim X-s$ with compact support on $X$, we put $\langle p^* S, \Phi \rangle:= \langle S, p_* \Phi \rangle$; see \cite[Page 18]{Demailly_ag}. If $S$ is closed and positive, then so is $p^* S$. Moreover if $S$ is a smooth form, then  the pull-back $p^* S$ coincide with the usual pull-back of $S$ as a form. It follows that if $S$ is equal to $\ddc v_1 \wedge \cdots \wedge \ddc v_k$ with $v_1,\ldots, v_k$ bounded psh functions on an open set $U$ in $Y$, then 
\begin{align}\label{eq-pullbackcurrent}
p^*S= \ddc (v_1 \circ p) \wedge \cdots \wedge \ddc (v_k \circ p)
\end{align}
on $p^{-1}(U)$. One can see it by regularizing $v_j$ and using the continuity of Monge-Amp\`ere operators under decreasing sequences. Similar arguments also show that if $S= \ddc v$ for a psh function $v$, then $p^* S= \ddc (v \circ p)$. 

In general when $p$ is no longer a submersion, there is no standard way to pull back a general closed positive current. Nevertheless for every closed positive $(1,1)$-current $S$ on $Y$, it is always possible by defining $p^* S:= \ddc (v \circ p)$, where $S= \ddc v$ locally for some psh function $v$. This definition is independent of the choice of $v$, and by above discussion, we recover the pull-back of $S$ if $p$ is a submersion.   
We will need the following auxiliary result. 

\begin{lemma} \label{le-pullbacksubmersion} Let $p: X \to Y$ be a  holomorphic submersion between two compact K\"ahler manifolds. Let $R$ be a closed positive $(1,1)$-current on $Y$. Then for every integer $l \ge 1$, we have $$p^* \langle R^l \rangle = \langle (p^* R)^l \rangle.$$
\end{lemma}

\proof We work locally on $Y$. Let $u$ be a local potential of $R$ on a local chart $U$ on $Y$, \emph{i.e.}, $R= \ddc u$. Thus $p^* R= \ddc (u \circ p)$ on $p^{-1}(U)$. One has 
$$\langle (p^* R)^l \rangle = \lim_{k \to \infty} \bold{1}_{p^{-1}\{u>-k\}} (\ddc  \max\{u \circ p, -k\} )^l$$
which is, by (\ref{eq-pullbackcurrent}), equal to 
$$ \lim_{k \to \infty}  p^* \big( \bold{1}_{\{u>-k\}} (\ddc  \max\{u, -k\} )^l \big)=p^* \langle R^l \rangle.$$ This finishes the proof. 
\endproof

Let $u$ be an $\omega_{n-1}$-psh function on $\P^{n-1}$.  In what follows, we will use the notations $\gtrsim, \lesssim$ to indicate, respectively,  $\ge, \le $ modulo strictly positive multiplicative constants independent of $u$.  Let $\P^n= \C^n \cup \P^{n-1}$  be the complex $n$-dimensional projective space, and $\omega_n$ the Fubini-Study form on $\P^n$. Using the coordinates $z$ on $\C^n$, we have $\omega_{n}= \frac{1}{2} \ddc \log (1+ \|z\|^2)$.  

Let $v$ be an $\omega_n$-psh function on $\P^n$ which is locally bounded outside $0 \in \C^n$ and the Lelong number of $v$ at $0$ is positive  (we identify $\C^n$ with an open subset in $\P^n$).  Put 
$$\varphi:= \max \{ \pi^* u+\log \|z\|,  v + \frac{1}{2} \log(1+\|z\|^2)+  \log \|z \|\}.$$
The particular case where $v:= \log \|z\| -  \frac{1}{2} \log(1+\|z\|^2)$ was considered in the Introduction.  Let $\lambda$ be the Lelong number of $v$ at $0$. By the choice of $v$, we get  $0< \lambda \le 1$.   Theorem \ref{the-chili-lelong} is a direct consequence of the following result applied to the case where $v=\log \|z\| -  \frac{1}{2} \log(1+\|z\|^2)$.   
 
\begin{theorem} \label{the-chili-lelong2v}   There is a constant $C>0$ independent of $u$ and $v$ such that
$$\lambda \, \tilde{e}_{k-1}(u) \le e_k(\varphi) -1 \le  C\tilde{e}_{k'}(u)$$
for every $1 \le k \le n$, where $k':= \min \{k, n-1\}$.  In particular, $e_n(\varphi)=1$ if and only if $u$ is of full Monge-Amp\`ere mass.
\end{theorem}

In the next paragraph we present a proof of Theorem  \ref{the-chili-lelong2v}.
We split it into several steps.  We first reformulate the question from a global point of view. \\

\noindent
\textbf{Globalization of the question.}  Put $v_1:=  \frac{1}{2}\log(1+ \|z\|^2)$. Let $H:= \P^n \backslash \C^n$.  We extend $-v_1$ trivially through $H$ (put $-v_1:= -\infty$ on $H$) to an $\omega_n$-psh function on $\P^n$.  We denote by $[H]$ the current of integration along $H$. Since the current $\ddc (-v_1)+ \omega_n$ is of mass $1$ and supported on $H$, we obtain 
\begin{align}\label{eq-tinhddctruv1}
\ddc (-v_1)= [H]-  \omega_n.
\end{align}
On the other hand, the function $\varphi':= \varphi - 2 v_1$ can be extended naturally to a global $(2 \omega_n)$-psh function on $\P^n$ which we still denote by $\varphi'$.  Put  
$$v_0:= \log \|z\|- v_1.$$
 We consider $v_0$ as an $\omega_n$-psh function on $\P^n$ rather than a function on $\C^n$. We have 
$$\varphi' = \max \{ \pi^* u- v_1 + v_0,  v+ v_0  \}.$$
Let  $T:= \ddc \varphi'+ 2\omega_n$. Note that $\varphi'$ is locally bounded outside $\{0\}$ (because $v$ and $v_0$ are so). Hence the intersection $T^k \wedge (\ddc v_0+ 2 \omega_n)^{n-k}$ is well-defined in the classical sense (see \cite[Chapter 3]{Demailly_ag}). It follows that 
 $$\int_{\P^n} T^k  \wedge (\ddc v_0+ 2 \omega_n)^{n-k}= \int_{\P^n} (2\omega_n)^n= 2^n.$$
Using again the fact that $\varphi'$ is locally bounded outside $\{0\}$ and the definition of non-pluripolar products, we obtain
\begin{align}\label{eq-tinhhigherlelongnumber}
e_k(\varphi) &= \int_{\{0\}} T^k \wedge (\ddc v_0+ 2 \omega_n)^{n-k}\\
\nonumber
&= \int_{\P^n} T^k  \wedge (\ddc v_0+ 2 \omega_n)^{n-k}- \int_{\P^n}  \langle T^k  \wedge (\ddc v_0+ 2 \omega_n)^{n-k} \rangle \\
\nonumber
&= 2^n - \int_{\P^n} \langle T^k  \wedge (\ddc v_0+ 2 \omega_n)^{n-k} \rangle.
\end{align}
Hence in order to estimate $e_k(\varphi)$, it suffices to compute the integral in the right-hand side of the last equality. To this end, we first need an idea from \cite{ChiLi}. Let $E:= \C^n$ and 
$$\widehat E:= \{(w, [w]): w \in E\backslash \{0\}\} \cup \{[w]: w \in E \backslash \{0\}\},$$
 where $[w]$ denotes the complex line passing through $w$. Notice that $\widehat E$ is the blowup of $E$ at $0$. Denote by $\rho: \widehat E \to E$ the natural projection sending $(w,[w])$ to $w$ for every $w \in E$.
 
Note that $\widehat E$ is naturally identified with the submanifold of $E \times \P^{n-1}$ defined by the equations $w_j y_s= w_s y_j$ for $1\le j,s \le n$, where $[y_1: \cdots: y_n]$ are the homogeneous coordinates on $\P^{n-1}$, and $w=(w_1,\ldots,w_n)$ are coordinates on $E$. Let 
$$U_j:= \{(w,[y]) \in \widehat E \subset E \times \P^{n-1}: y_j \not =0\}$$ for $1 \le j \le n$. Thus $(U_j)_{1 \le j \le n}$ is an open cover of $\widehat E$. We have canonical local coordinates on $U_j$. We describe those for $U_1$. The case of $U_j$ is done similarly.  Since $y_1 \not =0$ on $U_1$, natural local coordinates on $U_1$ are  $(w_1,y_2,\ldots, y_n)$. In these coordinates, one has 
\begin{align} \label{eq-bieudiendiaphuongcuarho} 
 \rho(w_1,y_2,\ldots, y_n)= (w_1, w_1 y_2,\ldots, w_1 y_n).
\end{align}

Observe that  the exceptional hypersurface $\widehat V:= \rho^{-1}(0)$ is naturally identified with $\P^{n-1}$. Moreover, we have a natural projection $\widehat p$ from $\widehat E \to \widehat V$ given by 
$$\widehat p(w, [w]):= [w].$$
The map $\widehat p$ makes $\widehat E$ to be a vector bundle of rank $1$ over $\widehat V$. Set $\overline{\widehat E}:= \P(\widehat E \oplus \C)$ and $\overline E:= \P(E \oplus \C)= \P^n$. The projection $\widehat p$ extends naturally to a projection from $\overline{\widehat E}$ to $\widehat V$ (which we still denote by $\widehat p$). Similarly, $\rho$ extends to a map from $\overline{\widehat E}$ to $\overline E$. Here is a  commutative diagram describing the map $\rho$:

\[ \begin{tikzcd}
\widehat V \subset \widehat E \arrow{r}{\imath} \arrow[swap]{d}{\rho} &  \overline{\widehat E} \arrow{d}{\rho} \\%
0 \in E \arrow{r}{\imath} & \overline{E}
\end{tikzcd}
\]
where the map $\imath$ denotes the natural inclusion map. 
Observe that 
$$\pi \circ \rho = \widehat p$$
 outside $\widehat V$. We have thus the following commutative diagram:

\[
  \begin{tikzcd}
    \widehat E \backslash \widehat V \arrow{r}{\rho} \arrow[swap]{dr}{\widehat p} & E\backslash \{0\} \arrow{d}{\pi} \\
     & \widehat V \approx \P^{n-1}
  \end{tikzcd}
\] 
Thus we obtain
\begin{align}\label{eq-pirho}
\rho^* \pi^* u= \widehat p^* u, \quad \rho^* \ddc \log \|z\| = \widehat p^* \omega_{n-1} \text{ on } \quad \widehat {E} \backslash \widehat V. 
\end{align}
\\

\noindent
\textbf{Proof of the desired upper bound for higher Lelong numbers in Theorem \ref{the-chili-lelong2v}.}
We now prove the desired upper bound for $e_k(\varphi)$.
Let $\psi_{\widehat V}$ be a global potential of $\widehat V$, \emph{i.e.,}
$$\ddc \psi_{\widehat V}+ \eta_{\widehat V} = [\widehat V],$$
for some closed smooth form $\eta_{\widehat V}$ in $\overline{\widehat E}$.   %Since the Lelong number of $\log \|z\|$ at $0$ is $1$, we have 
Write
$$\rho^* v_0= \psi_{\widehat V}+ \psi_0.$$
We check that $\psi_0$ is a smooth function on $\overline{\widehat E}$. It suffices to do it locally near $\widehat V$ because $v_0$ is smooth outside $\{0\}$ and $\rho^{-1}(0)= \widehat V$. 
We use the cover $(U_j)_j$ of $\widehat E$ as above. We check that $\psi_0$ is smooth on each $U_j$. We only do it for $j=1$ because the other cases are treated similarly. By (\ref{eq-bieudiendiaphuongcuarho}), in the local coordinates $(w_1,y_2,\ldots,y_n)$ of $U_1$, one gets 
\begin{align*}
\rho^* v_0(w_1,y_2,\ldots,y_n) &=\frac{1}{2} \log (|w_1|^2+ |w_1 y_2|^2+ \cdots+ |w_1 y_n|^2)- \rho^* v_1\\
&= \log |w_1|+\frac{1}{2} \log (1+ |y_2|^2+ \cdots+|y_n|^2)- \rho^* v_1.
\end{align*}
We infer that 
\begin{align}\label{eq-psi0nhan}
\psi_{\widehat V}+ \psi_0= \log |w_1|+\frac{1}{2} \log (1+ |y_2|^2+ \cdots+|y_n|^2)- \rho^* v_1
\end{align}
on $U_1$.  Since $\widehat V$ is given by $\{w_1=0\}$ on $U_1$, one sees that $\psi_{\widehat V}- \log |w_1|$ is a smooth function on $U_1$. This combined with (\ref{eq-psi0nhan}) and the fact that $v_1$ is smooth yields that $\psi_0$ is smooth.  Now compute
$$ \ddc \rho^* \log \|z\| = \ddc \rho^* v_1+ \ddc \rho^* v_0= \rho^* \omega_n+ \ddc \rho^* v_0=  \rho^* \omega_n+ \ddc \psi_0- \eta_{\widehat V}+ [\widehat V]$$
 which combined with  the second equality of (\ref{eq-pirho}) gives
 $$\rho^* \omega_n+ \ddc \psi_0- \eta_{\widehat V}= \widehat p^* \omega_{n-1}$$
 on $\widehat{E} \backslash \widehat V$ (note $[\widehat V]$ vanishes outside $\widehat V$). Since both sides of the last equality are smooth forms on $\overline{\widehat{E}}$ and  $\widehat{E} \backslash \widehat V$ is dense in $\overline{\widehat{E}}$, we obtain
\begin{align}\label{tinh-etawidehatddcpsi0}
  \rho^* \omega_n+ \ddc \psi_0- \eta_{\widehat V}= \widehat p^* \omega_{n-1}
  \end{align}
  on $\overline{\widehat{E}}$.
  Put $\theta:= \widehat p^* \omega_{n-1}+  \rho^* \omega_n$ and $\widehat T:= \rho^* T$.  By (\ref{tinh-etawidehatddcpsi0}), one gets
\begin{align}\label{tinh-etawidehatddcpsi0_them}
\rho^*\ddc v_0+  2 \rho^* \omega_n= [V]- \eta_{\widehat V}+ \ddc \psi_0+  2 \rho^* \omega_n= \theta+[V]
\end{align}
Combining (\ref{tinh-etawidehatddcpsi0_them}) with the fact that the mass of  non-pluripolar products of maximal order is preserved under pull-backs of a smooth modification  yields  
\begin{align}\label{eq-tinhmassTn}
\int_{\overline E} \langle T^k  \wedge (\ddc v_0+ 2 \omega_n)^{n-k} \rangle &= \int_{\overline{\widehat E}} \langle (\rho^* T)^k  \wedge \rho^*(\ddc v_0+ 2 \omega_n)^{n-k} \rangle\\
\nonumber
&= \int_{\overline{\widehat E}} \langle \widehat T^k  \wedge (\theta+[V])^{n-k} \rangle= \int_{\overline{\widehat E}} \langle \widehat T^k  \wedge \theta^{n-k} \rangle,
\end{align} 
where in the last equality we used the multi-linearity of non-pluripolar products and (\ref{eq-bang0nonpluripolarD}). 
%Put $$A:= \widehat p^{-1} \{u = -\infty\} \cup \widehat V,$$where we identify $\widehat V$ with $\P^{n-1}$. The last set is pluripolar on $\overline{\widehat E}$. 
By (\ref{tinh-etawidehatddcpsi0_them}) and (\ref{eq-pirho}), we get 
\begin{align*}
\widehat T  &= \rho^* \ddc \max\{\pi^* u- v_1,  v \}+ \rho^*\ddc v_0+  2 \rho^*\omega_n\\
&=\ddc \max\{\widehat p^* u - \rho^* v_1, \rho^* v\}+ \theta+ [V].
\end{align*}
Hence 
\begin{align}\label{eq-congthucTmu}
\langle \widehat T^k \rangle = \big\langle (\ddc \max\{\widehat p^* u - \rho^* v_1, \rho^* v\}+ \theta)^k \big\rangle.
\end{align}
Note that both $\widehat p^* u - \rho^* v_1$ and $\rho^* v$ are $\theta$-psh functions. Thus $\max\{\widehat p^* u - \rho^* v_1, \rho^* v\}$ is also $\theta$-psh. 
Put $R:= \ddc u+ \omega_{n-1}$.  Observe that 
$$\max\{\pi^* u - v_1, v\} \ge \max\{ \pi^* u,  v\}+ O(1) \ge  \pi^* u + O(1)$$
on $E$ because $\pi^* u$ is bounded from above, and $v$ is locally bounded outside $0$. We infer that 
$$\max\{\widehat p^* u - \rho^* v_1, \rho^* v\} \ge \pi^* u+ O(1).$$
 Using this and the monotonicity of non-pluripolar products (Theorem \ref{th-mono-current11}) applied to $\ddc \max\{\widehat p^* u - \rho^* v_1, \rho^* v\}+ \theta$ and $\widehat p^* R$  yields
\begin{align}\label{ine-congthemkahler}
\int_{\widehat{\overline E}} (\theta^n - \langle \widehat T^k \rangle \wedge \theta^{n-k})  &\le \int_{\widehat{\overline E}} \bigg[\big( (\widehat p^* \omega_{n-1}+ \rho^* \omega_n)^k- \langle (\widehat p^* R+ \rho^* \omega_n)^{k}\rangle \big) \wedge \theta^{n-k}\bigg]\\
\nonumber
& = \sum_{j=0}^k   \binom{k}{j} \int_{\widehat{\overline E}}    \big( \widehat p^* \omega_{n-1}^j - \langle (\widehat p^* R)^j \rangle \big)\wedge \rho^* \omega_n^{k-j}  \wedge \theta^{n-k}.
\end{align}
Denote by $A_j$ the integral in the $j$-th summand of the right-hand side of the last inequality.  By Lemma \ref{le-pullbacksubmersion}, one gets 
$$A_j=  \int_{\widehat{\overline E}}\big( \widehat p^* \omega_{n-1}^j - \widehat p^* \langle R^j \rangle \big)\wedge \rho^* \omega_n^{k-j}  \wedge \theta^{n-k}.$$
Recall $k'=\min\{k,n-1\}$. Observe that for $j > k'$, one has $A_j = 0$ because both $\omega_{n-1}^n$ and $\langle R^n \rangle$ vanish (they are currents on $\P^{n-1}$). On the other hand, by the definition of the pull-back $\widehat p^*$, for $0 \le j \le k'$,  we get
\begin{align*}
A_j &= \int_{\P^{n-1}} \big(\omega_{n-1}^j - \langle R^j \rangle \big)\wedge \widehat p_*(\rho^* \omega_n^{k-j}  \wedge \theta^{n-k})\\
& \lesssim A'_j:= \int_{\P^{n-1}} \big(\omega_{n-1}^j - \langle R^j \rangle \big)\wedge \omega_{n-1}^{n-1-j},
\end{align*}
where we used the fact that $\{\omega_{n-1}^j \} \ge \{\langle R^j \rangle \}$ (a consequence of Theorem \ref{th-mono-current11}). 
For $0 \le j \le k'$, consider
\begin{align*}
A'_j- \tilde{e}_{k'}(u)&=  \int_{\P^{n-1}} \langle R^{k'} \rangle \wedge \omega_{n-1}^{n-1-k'}- \int_{\P^{n-1}} \langle R^{j} \rangle \wedge \omega_{n-1}^{n-1-j} \le 0
\end{align*}
by Theorem \ref{th-mono-current11} again and the fact that $R$ is more singular than $\omega_{n-1}$. Combining this and (\ref{ine-congthemkahler}) implies 
\begin{align}\label{ine-congthemkahlersualai}
\int_{\widehat{\overline E}} (\theta^n - \langle \widehat T^k \rangle \wedge \theta^{n-k})  \lesssim  \tilde{e}_{k'}(u).
\end{align}
Since 
$$\theta:=\ddc \psi_0 - \eta_{\widehat V}+ 2 \rho^* \omega_n,$$
we can express
\begin{align*}
\int_{\overline{\widehat E}} \theta^n  &=\int_{\overline{\widehat E}} (2 \rho^* \omega_n- \eta_{\widehat V})^n=  \int_{\overline{\widehat E}} 2 \rho^* \omega_n \wedge \theta^{n-1} - \int_{\overline{\widehat E}} \eta_{\widehat V} \wedge \theta^{n-1}
\end{align*}
Denote by $I$ the  second integral in the right-hand side of the last equality. Since $\rho^* \omega_n|_{\widehat V}= \rho^* (\omega_n|_{\{0\}}) =0$, we infer that  
$$I=  \int_{\widehat V} \theta^{n-1}= \int_{\widehat V} (-\eta_{\widehat V})^{n-1}= 1.$$
Hence 
\begin{align}\label{eq-tinhmassofthreta}
\int_{\overline{\widehat E}} \theta^n = \int_{\overline{\widehat E}} 2 \rho^* \omega_n \wedge \theta^{n-1} -1 = \cdots =  \int_{\overline E} (2\omega_n)^n -1 = 2^n -1.
\end{align}
Combining this,  (\ref{eq-tinhhigherlelongnumber}), (\ref{eq-tinhmassTn}), (\ref{ine-congthemkahlersualai}) and (\ref{eq-tinhmassofthreta}) yields that 
$$e_k(\varphi) - 1 \lesssim \tilde{e}_{k'}(u).$$
\\

\noindent
\textbf{Proof of the desired lower bound for higher Lelong numbers in Theorem \ref{the-chili-lelong2v}.}
We check the desired lower bound for $e_k(\varphi)$.  Since $\lambda$ is the Lelong number of $v$ at $0$, we get 
$$v \le \lambda \log \|z\|+ O(1)$$
 near $0$. Using this and Demailly's comparison of Lelong numbers, we see that it suffices to prove the desired inequality in the case where 
$$v= \lambda \big(\log \|z\|- \frac{1}{2} \log(1+ \|z\|^2)\big)= \lambda v_0.$$
Set 
$$I_j:= \int_{\widehat{\overline E}} \langle \widehat p^* R^j  \wedge   (\ddc \rho^* v+ \theta)^{k-j} \rangle \wedge \theta^{n-k} $$
and
$$I'_j:= \int_{\widehat{\overline E}} \langle \widehat T \wedge \widehat p^* R^j \wedge (\ddc \rho^* v+ \theta)^{k-1-j}  \rangle \wedge \theta^{n-k}$$
for $0 \le j \le k$. By (\ref{eq-tinhddctruv1}), we obtain
$$\widehat p^* R= \ddc (\widehat p^* u - \rho^* v_1)+ \rho^* \omega_n+ \widehat p^* \omega_{n-1}- \rho^* [H]=  \ddc (\widehat p^* u - \rho^* v_1)+ \theta- \rho^* [H].$$
By this and (\ref{eq-congthucTmu}) and Lemma \ref{le-ddmax-blocki}, we get  
\begin{align*}
I &:= \int_{\widehat{\overline E}} \langle \widehat T^k \rangle \wedge \theta^{n-k} \\
&=\int_{\widehat{\overline E}} \big\langle (\ddc \max\{\widehat p^* u - \rho^* v_1, \rho^* v\}+ \theta)^k \big\rangle \wedge \theta^{n-k}=   \sum_{j=0}^{k-1} I'_j - \sum_{j=1}^{k-1} I_j.
\end{align*}
We estimate $I_j,I'_j$.  %We have $$\rho^* v = \lambda \psi_{\widehat V}+ \psi$$ (sum of two quasi-psh functions).We note that $\psi$ is not bounded in general. Using the fact that $\ddc v+ \omega_n \ge 0$ gives $$\ddc \psi- \lambda \eta_{\widehat V}+ \rho^* \omega_n \ge 0.$$ 
By the multi-linearity of non-pluripolar products and (\ref{eq-bang0nonpluripolarD}) (and remember that $v= \lambda v_0$), for $1 \le j \le n-1$,  we have 
\begin{align} \label{tinhIjchanduoi}
I_j  &= \int_{\widehat{\overline E}} \langle \widehat p^* R^j  \wedge   (\lambda \ddc \psi_{\widehat V}+ \lambda \ddc \psi_0+ \theta)^{k-j} \rangle \wedge \theta^{n-k}\\ 
\nonumber
&= \int_{\widehat{\overline E}}  \langle \widehat p^* R^j  \wedge   (\lambda [\widehat V] - \lambda \eta_{\widehat V}+  \theta)^{k-j} \rangle \wedge \theta^{n-k} \\
\nonumber
&=  \int_{\widehat{\overline E}} \langle \widehat p^* R^j  \wedge   (\theta - \lambda\eta_{\widehat V})^{k-j} \rangle \wedge \theta^{n-k} \\
\nonumber
&=  \int_{\widehat{\overline E}} \langle \widehat p^* R^j  \wedge   (\theta - \lambda\eta_{\widehat V})^{k-1-j} \wedge (- \lambda \eta_{\widehat V}) \rangle \wedge \theta^{n-k} + \int_{\widehat{\overline E}} \langle \widehat p^* R^j \rangle  \wedge   (\theta - \lambda \eta_{\widehat V})^{k-1-j} \wedge \theta \wedge \theta^{n-k}  \\
\nonumber
&=- \lambda \int_{\P^{n-1}} \langle  R^j \rangle  \wedge   ( (1+ \lambda) \omega_{n-1})^{k-1-j} \wedge \omega_{n-1}^{n-k} +\int_{\widehat{\overline E}} \langle \widehat p^* R^j \rangle  \wedge   (\theta - \lambda \eta_{\widehat V})^{k-1-j} \wedge  \theta^{n-k+1},
\end{align}
here we used the fact that  $\eta_{\widehat V}|_{\widehat V} $ is cohomologous to $-\omega_{n-1}$, and $\theta|_{\widehat V}$ is cohomologous to  $- \eta_{\widehat V}|_{\widehat V}$. 

We  treat $I'_j$. Observe  that 
$$-\lambda \eta_{\widehat V}+ \rho^* \omega_n+ \lambda\ddc \psi_0= \lambda (\ddc \psi_0- \eta_{\widehat V}+ \rho^* \omega_n)+ (1- \lambda) \rho^* \omega_n  \ge 0$$
by (\ref{tinh-etawidehatddcpsi0}).
Hence the cohomology class of $-\lambda \eta_{\widehat V}+ \theta$ is semi-positive.     Using this and  the monotonicity of non-pluripolar products again, we obtain
\begin{align*} %\label{tinhIjchanduoirelativeNO}
I'_j  & = \int_{\widehat{\overline E}}  \langle \widehat T  \wedge (-  \lambda \eta_{\widehat V}+ \theta)^{k-1-j} \wedge \widehat p^* R^j \rangle \wedge \theta^{n-k}\\ 
\nonumber
& \le \int_{\widehat{\overline E}} \theta \wedge (\theta- \lambda \eta_{\widehat V})^{k-1-j} \wedge \langle \widehat p^* R^j \rangle \wedge \theta^{n-k}.
\end{align*}
It follows that 
\begin{align*}
I  &\le    \int_{\widehat{\overline E}}  \theta^{n-k+1} \wedge (\theta-  \lambda \eta_{\widehat V})^{k-1}
+ \quad   \sum_{j=1}^{k-1}  \lambda \int_{\P^{n-1}} \langle R^j\rangle  \wedge   ((1+\lambda)\omega_{n-1})^{k-1-j}\wedge \omega_{n-1}^{n-k}\\
& =  \int_{\widehat{\overline E}} \theta^{n}
+ \quad   \lambda \sum_{j=1}^{k-1} (1+\lambda)^{k-1-j} \int_{\P^{n-1}} \big(\langle R^j\rangle \wedge  \omega_{n-1}^{k-1-j} -  \omega_{n-1}^{k-1} \big)\wedge \omega_{n-1}^{n-k}.
%& \le \int_{\widehat{\overline E}} (\theta+\widehat p^* \omega_{n-1})^{n}+ \quad  \int_{\P^{n-1}} \big(\langle (\widehat p^* R+\theta)^j\rangle - (3 \omega_{n-1})^{n-1} \big).
\end{align*}
Consequently,
$$\int_{\widehat{\overline E}} (\theta^n - \langle \widehat T^k \rangle \wedge \theta^{n-k}) \ge  \lambda \int_{\P^{n-1}} \big(\omega_{n-1}^{k-1} -\langle  R^{k-1}\rangle  \big) \wedge \omega_{n-1}^{n-k} =  \lambda \tilde{e}_{k-1}(u).$$
This completes the proof of Theorem \ref{the-chili-lelong2v}.

We thus have finished the proof of Theorem \ref{the-chili-lelong2v}. We now deal with Corollary \ref{cor-phanvidu}.

For every psh function $w$ on an open subset $U$ of $\C^n$, recall that the multiplier ideal sheaf $\cali{I}(m w)$  is the ideal sheaf of holomorphic germs $f$ at a point $z \in U$ such that $|f|^2 e^{-2m w}$ is locally integrable with respect to the Lebesgue measure on $\C^n$ near $z$. 

Let $\psi$ be a psh function on the unit ball $\B$ in $\C^N$ ($\psi$ is not necessarily locally bounded outside $0$). For $m \in \N$, let $(\sigma_{jm})_{j \in \N}$ be a basis of the Hilbert space $\cali{H}_{m\psi}(\B)$ of holomorphic functions $f$ on $\B$ such that 
$$ \|f\|:= \int_{\B} |f|^2 e^{-2m \psi} d \, vol < \infty,$$
 where $vol$ is the Lebesgue measure on $\C^N$. It is a well-known fact that $(\sigma_{jm})_j$ generates the multiplier ideal sheaf $\cali{I}(m \psi)$ (see \cite[Proposition 5.7]{Demailly_analyticmethod}). Put 
$$\psi_m:= \frac{1}{2m} \log \sum_{j=1}^\infty |\sigma_{jm}|^2 = \frac{1}{m}\sup_{f \in \cali{H}_{m \psi}(\B): \|f\| \le 1}\log |f|.$$ 
Since $\cali{I}(m \psi)$ is a coherent sheaf (see \cite[Page 37]{Demailly_analyticmethod}), it is locally generated by finitely many holomorphic germs. Hence we see that $\psi_m$ has analytic singularities on $\B$.  By Demailly's analytic regularization of psh functions (\cite[Theorem 13.2]{Demailly_analyticmethod}), we have $\psi_m \ge \psi - C/m$ for some constant $C>0$ independent of $m$, and the sequence $(\psi_m)_m$ converges to $\psi$ in $L^1_{loc}$ as $m \to \infty$  and their Lelong numbers also converge to those of $\psi$. In general the sequence is not decreasing as pointed out in \cite{DanoKim-remark}.  Moreover, as shown by the following lemma, to some extent,  $(\psi_m)_m$ is indeed the best sequence to approximate $\psi$ among those having analytic singularities and less singular than $\psi$. 

\begin{lemma} \label{le-demailly-analytic} Let $\psi'$ be a psh germ at $0$ having analytic singularities such that $\psi'$ is less singular than $\psi$. Assume that $\psi$ is locally bounded outside $\{0\}$. Then we have  
$$e_k(\psi') \le \sup_{m \in \N} e_k(\psi_m) \le e_k(\psi)$$
for every $1 \le k \le N$.
\end{lemma}

Note that the higher Lelong numbers at $0$ are well-defined because of the assumption that  $\psi$ is locally bounded outside $\{0\}$.

\proof   Without loss of generality, we can assume that $\psi'$ is defined on $\B$. Let $(\psi'_m)_m$ be the sequence of psh functions with analytic singularities on $\B$ given by Demailly's analytic approximation theorem for $\psi'$. Since $\psi'$ has analytic singularities, by comparison of Lelong numbers (\cite[Page 166]{Demailly_ag}), we can assume that $\psi'= c \log \sum_{j=1}^l |f_j|^{\alpha_j}$, where $f_1,\ldots, f_l$ are holomorphic on an open neighborhood of $0$ and $\alpha_1,\ldots, \alpha_m,c$ are nonnegative constants. In particular $e^{\psi'}$ is H\"older. By this and  \cite[Lemma 5.10]{Boucksom-Favre-Jonsson} (see also \cite{R-approx-psh}), there exists a constant $C>0$ independent of $m$ such that  
$$\psi' \le \psi'_m + O(1), \quad \psi'_m \le (1- C/m)\psi'+ O(1).$$
It follows that 
$$e_k(\psi'_m) \to e_k(\psi')$$
 as $m \to \infty$. On the other hand, since $\psi \le \psi' + O(1)$, we infer $\psi_m \le \psi'_m + O(1)$. Hence, $e_k(\psi_m) \ge e_k(\psi'_m)$. Letting $m \to \infty$ yields
$$\limsup_{m \to \infty} e_k(\psi_m) \ge e_k(\psi').$$
This finishes the proof. 
\endproof

\begin{lemma} \label{le-idealsheaf} Let $w$ be a psh function on an open subset $U$ on $\C^n$. If $u$ has zero Lelong number everywhere, then we have 
$$\cali{I}_0(m(w +\log \|z\|))= \cali{I}_0 (m (\pi^* u+ w+\log \|z\|)).$$
\end{lemma}

We refer to \cite[Corollary 2.2]{ChiLi} for a similar statement when $w=0$, and also to \cite[Proposition 2.3]{DanoKim-siumetric} for a closely related result.   We deduce from the above lemma that
$$\cali{I}_0(m(\log\|z\|))= \cali{I}_0(m \varphi)$$
because  
$$\pi^* u+ \log \|z\| \le \varphi \le \log \|z\|+ C$$
near $0$, for some constant $C>0$.

\proof  Put $w':= w+ \log \|z\|$. Since $\pi^* u+w' \le w' +O(1)$, we get 
$\cali{I}_0 (m (\pi^* u+ w')) \subset \cali{I}_0(m w')$. We prove the converse inclusion. Let $\rho, \widehat E, \widehat V, \widehat p$ be as in the proof of Theorem \ref{the-chili-lelong2v}. Recall that $\widehat p = \pi \circ \rho$ outside $\widehat V$.  Let $f \in \cali{I}_0 (m w')$, \emph{i.e.,}
$$\int_U |f|^2 e^{-2m w'} \, vol_{\C^n}<\infty,$$
for some open subset $U$ containing $0$. Let $\mu:= \rho^* vol_{\C^n}$. We infer that 
$$\int_{\rho^{-1}(U)} |f \circ \rho|^2 e^{-2m  w' \circ \rho} \, d \mu < \infty.$$
Direct computations show  that $\mu= |g|^2 \vol_{\C^n}$ locally near every point in $\widehat V$ for some local holomorphic function $g$. 
By this and  the strong openness conjecture (which is now a theorem, see \cite{Guan-Zhou-strong-openness} and also \cite{Hiep-openness}), there exists a constant $\epsilon>0$ such that 
\begin{align}\label{ine-openness}
\int_{\rho^{-1}(U)} |f \circ \rho|^2 e^{-2(m+\epsilon)  w' \circ \rho} \, d \mu < \infty
\end{align}
(we shrink $U$ if necessary). On the other hand, since $u$ has zero Lelong number everywhere, Skoda's estimate (\cite[Lemma 5.6]{Demailly_analyticmethod} or \cite{Skoda-lelong})  implies that 
\begin{align} \label{ine-skoda}
\int_{\rho^{-1}(U)} e^{-2 q (u \circ \widehat p)} \, d \mu < \infty
\end{align}
for every constant $q>0$ (we shrink $U$ if necessary). This combined with H\"older's inequality yields that 
\begin{multline*}
\int_{\rho^{-1}(U)} |f \circ \rho|^2  e^{-2m( (\pi^*u) \circ \rho + w' \circ \rho)} \, d \mu\le   \bigg( \int_{\rho^{-1}(U)} |f \circ \rho|^2 e^{-2(m+ \epsilon)w' \circ \rho} \, d \mu \bigg)^{\frac{m}{m+\epsilon}} \times \\
\bigg( \int_{\rho^{-1}(U)} |f \circ \rho|^2 e^{-2 \frac{m(m+\epsilon)}{\epsilon} u \circ \widehat p} \, d \mu \bigg)^{\frac{\epsilon}{m+\epsilon}}.
\end{multline*} 
This combined with (\ref{ine-skoda}) and (\ref{ine-openness}) gives
$$\int_{\rho^{-1}(U)} |f \circ \rho|^2  e^{-2m( (\pi^*u)\circ \rho + w' \circ \rho)} \, d \mu < \infty.$$
Hence $f \in \cali{I}_0 (m (\pi^* u+ w'))$.  The desired assertion follows. This finishes the proof.  
\endproof

As kindly pointed out to us by one of referees, %one can deduce Lemma \ref{le-idealsheaf} as follows (this argument is indeed simpler than our above proof).  
by using similar arguments as above with the strong openness theorem and H\"older inequality, one can also show  that 
\begin{align}\label{ine-refereeremark}
\cali{I}_0(m(w +\log \|z\|))= \cali{I}_0 (m (\phi+ w+\log \|z\|)),
\end{align}
for any psh function $\phi$ having zero Lelong number at $0$. We note however that Lemma \ref{le-idealsheaf} does not follow directly from (\ref{ine-refereeremark}) applied to $\phi= \pi^* u$ because the function $\pi^* u$ is not psh near $0$ (but $\pi^* u+ \log \|z\|$ is so). %  Now observe that $\pi^* u$  has zero Lelong number at $0$ because otherwise one would have $u(\pi(z))= u(\pi(\lambda z)) \to -\infty$ as $\lambda \in \C \backslash \{0\} \to 0$, hence, $u \equiv -\infty$ (contradiction).  Hence one can apply (\ref{ine-refereeremark}) to $\phi:= \pi^* u$ to obtain Lemma \ref{le-idealsheaf}.

\begin{proof}[Proof of Corollary \ref{cor-phanvidu}] Since $\ddc u+ \omega_{n-1}$ has mass on some pluripolar set,  one gets $\tilde{e}_1(u)>0$. Consequently,  $\tilde{e}_{k}(u)>0$ for every $1 \le k \le n-1$ by (\ref{ine-ehigher}). Theorem \ref{the-chili-lelong2v} then shows that for $2 \le k \le n$, one has  $e_k(\varphi)>1$.

Since  $\varphi(z) \le \log \|z\|+C$ for some constant $C>0$, one gets 
\begin{align}\label{inebosung}
\sup_{\varphi' \in \cali{A}_\varphi} e_k(\varphi') \ge  e_k(\log \| z \|) =1.
\end{align}
On the other hand, let $(\varphi_m)_m, (\psi_m)_m$ be the sequence of psh functions with analytic singularities associated to $\varphi$, $\log \|z\|$ respectively as in the paragraph before Lemma \ref{le-demailly-analytic}. Since   $\varphi$ and $\log \|z\|$ are valuatively equivalent at $0$ by Lemma \ref{le-idealsheaf}, one sees that  $\varphi_m$ and $\psi_m$ are of the same singularity type. Hence 
$$e_k(\varphi_m) = e_k(\psi_m) \le e_k(\log \|z\|) =1.$$ 
% Since $e^{\log \|z\|}$ is H\"older, by \cite[Lemma 5.10]{Boucksom-Favre-Jonsson} again, one gets $$|e_k(\psi_m)- e_k(\log \|z\|)| \le C/m$$
%for some constant $C>0$ independent of $m$.  Consequently we obtain$$|e_k(\varphi_m)-1| \le C/m.$$
By this and Lemma \ref{le-demailly-analytic}, there holds
$$\sup_{\varphi' \in \cali{A}_\varphi} e_k(\varphi') \le  \sup_{m \in \N} e_k(\varphi_m) \le 1.$$
This combined with (\ref{inebosung}) gives
$$\sup_{\varphi' \in \cali{A}_\varphi} e_k(\varphi') = 1.$$ 
The desired assertion thus follows.  
\end{proof}

\begin{remark}
One can obtain a slightly more general version of  Theorem \ref{the-chili-lelong2v} as follows. Let $v'$ be a psh function on an open neighborhood of $0 \in \C^n$ such that $v' - v =O(1)$ near $0$. Then  for $\varphi':= \max \{ \pi^* u+\log \|z\|,  v'+  \log \|z \|\}$, we have $e_k(\varphi')= e_k (\varphi)$ by Demailly's comparison of Lelong numbers. Hence Theorem \ref{the-chili-lelong2v} is still true for $\varphi'$ in place of $\varphi$. 
\end{remark}

\section{Quasi-psh functions with zero Lelong number everywhere} \label{sec-zerolelong}

In this section, we explain how to construct many examples of closed positive $(1,1)$-currents $T$ on a projective manifold such that $T$ has zero Lelong number everywhere and $T$ charges some pluripolar set.

We recall the following classical result; see \cite[Theorem 3, Page 31] {Carleson-small-set} (or see \cite{ChiLi} for an explicit construction).

\begin{lemma} \label{pro-Evanstheorem} There exist an uncountable compact polar set $A$ (which is also a generalized Cantor set) on $\C$  and a Borel probability measure $\mu$ supported on  $A$ such that $\mu$ has no atom, and hence the potential of $\mu$ has zero Lelong number everywhere.  
\end{lemma}

In the above result, we recall that the potential of $\mu$ is given by $p_\mu(z):= \int_\C \log |z-w| d\mu(w)$. Since $\mu$ is supported on a compact subset in $\C$, one sees that $p_\mu(z) - \frac{1}{2}\log(1+|z|^2)$ is a bounded function outside a disk of sufficiently large radius in $\C$. It follows that $p_\mu-\frac{1}{2}\log(1+|z|^2)$ can be extended to an $\omega_1$-psh function on $\P^1$ (where $\omega_1$ is the Fubini-Study form on $\P^1$) which has zero Lelong number everywhere and the current $\ddc (p_\mu-\frac{1}{2}\log(1+|z|^2)) +\omega_1$ has positive mass on a compact polar set.  

The following auxiliary result is a direct consequence of \cite[Theorem 5.1]{Kiselman-sublevel-set} or \cite[Theorem 2]{Favre-pullback-lelong}.

\begin{lemma}\label{le-pullbacklelong} Let $f: X_1 \to X_2$ be a surjective holomorphic map between connected complex manifolds. Let $T$ be a closed positive current of bi-degree $(1,1)$ on $X_2$ having zero Lelong number everywhere. Then the Lelong number of $f^* T$ at every point in $X_1$ is also zero.  
\end{lemma}

We also need the well-known fact about the Lelong number of the direct image of currents.

\begin{lemma}\label{le-pullbacklelong2} Let $f: X_1 \to X_2$ be a finite proper holomorphic map between complex manifolds. Let $T$ be a closed positive current of bi-degree $(1,1)$ on $X_1$ having zero Lelong number everywhere. Then the Lelong number of $f_* T$ at every point in $X_2$ is also zero.  
\end{lemma}

\proof This is a direct consequence of \cite[Theorem 9.12]{Demailly_ag}. One can also argue directly as follows. Express $T= \ddc u+ \theta$, where $\theta$ is a smooth closed form and $u$ is a $\theta$-psh function. By hypothesis, the Lelong number of $u$ is zero everywhere. We have $f_*T= \ddc f_* u+ f_* \theta$. Since $f$ is finite, the form $f_* \theta$ is continuous. Hence $f_*u$ is a quasi-psh function on $X_2$, and $\nu(f_*u, x)= \nu(f_*T,x)$ for every $x \in X_2$. Let $K_j$ be a compact subset in $X_j$ for $j=1,2$ such that $f^{-1}(K_2) \subset K_1$. By subtracting a constant from $u$, we can assume $u$ is negative on $K_1$. Using Skoda's estimate (\cite{Skoda_integrability}) and the fact that $\nu(u,\cdot)\equiv 0$ implies that 
$$\int_{K_1} e^{- M u} d\, vol_1 < \infty$$   
 for every constant $M>0$, and $vol_1$ is a smooth volume form on $X_1$. It follows that for every smooth form $vol_2$ on $X_2$ one gets
$$\int_{K_2} e^{- M f_* u} d \, vol_2 \le \int_{K_1} e^{- M  u} d\, vol_1 <\infty$$
for every constant $M>0$. Hence $\nu(f_*u, x)=0$ for every $x \in X_2$. This finishes the proof. 
\endproof

\begin{proposition} \label{pro-Tzeropluri} Let $X$ be a projective manifold. Then there exists a closed positive current $T$ of bi-degree $(1,1)$ on $X$ such that $T$ has zero Lelong number everywhere and $T$ charges some pluripolar set, \emph{i.e.,} the trace measure of $T$ has positive mass on some pluripolar set. 
\end{proposition}

\proof  Let $n:= \dim X$. We first prove the desired assertion when $X= \P^n$.  Let $T_1$ be a current on $\P^1$ such that $T_1$ has zero Lelong number everywhere and $T_1$ has mass on some compact pluripolar set. Let $p_1: (\P^1)^n \to \P^1$ be the natural projection to the first component of $(\P^1)^n$. Let $T_1':= p_1^* T_1$. Hence, $T'_1$ charges some compact pluripolar set $A$ on $(\P^1)^n$. Moreover, by Lemma \ref{le-pullbacklelong},  we see that $T_1'$ has zero Lelong number everywhere. 

Consider now a finite map $f: (\P^1)^n \to \P^n$ such that $f$ is a covering map outside some proper analytic subset $Y$ of $(\P^1)^n$, and $T'_1$ has mass on $A \backslash Y$  (such maps exist abundantly; to see it, one just needs to embed $(\P^1)^n$ into a projective space $\P^N$ and use natural projections from $\P^N$ to $\P^n$).   By Lemma \ref{le-pullbacklelong2} and properties of $T'_1$, the current $T''_1: = f_* T'_1$ has zero Lelong number everywhere. Using now the fact that $f$ is a covering on $(\P^1)^n \backslash Y$, we see that the mass of $T''_1$ on  $f(A) \backslash f(Y)$ is positive. Since $f(A)$ is pluripolar, we deduce that $T''_1$ charges the compact pluripolar set $A'':=f(A)$. Hence we have proved the desired assertion for $X= \P^n$.  
\\

\noindent
\textbf{Claim.}  For every (not necessarily nonsingular) algebraic variety $X$, there exists a finite holomorphic map $g: X \to \P^n$ such that $g$ is a covering map outside some proper analytic set $Y_1 \subset X$ and $T''_1$  has mass on $A'' \backslash g(Y_1)$. 
\\

\noindent
Assume the Claim holds for the moment. We explain how to finish the proof. Lemma \ref{le-pullbacklelong} implies that $T:= g^* T''_1$ has zero Lelong number everywhere. We also see that $T$ has mass on the pluripolar set $g^{-1}(A'')$ because $T''_1$ charges $A'' \backslash g(Y_1)$ and $g$ is a covering map on $X \backslash Y_1$. 

It remains to check Claim.  We consider $X$ as a variety of dimension $n$ in $\P^N$. We do it by induction on $m:=N-n$. If $m=0$, there is nothing to prove. Suppose now the claim holds for every variety of dimension $n+1$ in $\P^N$. We now check it for $X$ of dimension $n$. Let $a\in \P^N \backslash X$ and $H$ be a hyperplane in $\P^N$ which does not contain $a$. Let $p_a: \P^N\backslash \{a\} \to H$ be the natural projection with center $a$ to $H$. Recall that $p_a$  is defined as follows: for every $b \in \P^N\backslash \{a\}$, the point $p_a(b)$ is the intersection of $H$ with the projective line passing through $a$ and $b$. 

Since $a$ is not in $X$, the fiber of $p_a$ cuts $X$ at isolated points. Thus the restriction of $p_a$ to $X$ is a finite map, and by compactness of $X$, the image $Y:=p_a(X)$ is again a variety of dimension $n$ in $H \approx \P^{N-1}$. Applying now the induction hypothesis to $Y$ in $H$, we obtain a finite map $p_Y: Y \to \P^n$.  Define $\tilde{g}:= p_Y \circ p_a$. One sees that $\tilde{g}$ is a finite map from $X$ to $\P^n$.  Let $Y_1$ be the set of critical points of $\tilde{g}$ which is a proper analytic set in $X$. Let $b\in A''$ be a point such that $T''_1$ charges every open neighborhood of $b$ in $A''$. Such a point exists because otherwise for every point in $A''$ one could find an open neighorhood of that point in $A''$ such that $T''_1$ has no mass there. Hence, one obtains an open cover $(U_j)_j$ of $A''$ and $T''_1$ has no mass on $U_j$ for every $j$. Using this and the compactness of $A''$, we can find finitely many $U_{j_1}, \ldots, U_{j_s}$ covering $A''$ such that $T''_1$ has no mass on $U_{j_t}$ for $1 \le t \le s$, this is a contradiction because the mass of (the trace measure of) $T''_1$ on $A''= \cup_{1 \le t \le s} U_{j_t}$ is positive (one actually does not need to use the compactness of $A''$ because by the Lindel\"of property every open cover of $A''$ admits a countable subcover).  

Now using the transitivity of the automorphism group of $\P^n$, there is  an (linear) autormorphism $q$ of $H$ so that  $q(\tilde{g}(Y_1))$ does not contain $b$.  Put $g:= q \circ \tilde{g}$.   By the choice of $q$, $T''_1$ charges the complement in $A''$ of $g(Y_1)=q(\tilde{g}(Y_1))$ because the latter (closed) set does not contain $b$. Claim follows.
This finishes the proof.
\endproof

\begin{corollary} \label{cor-currentvanishinlelong}
Let $X$ be a projective manifold. Then there exist a K\"ahler form $\eta$ and  a closed positive current $T$ cohomologous to $\eta$ such that $T$ has zero Lelong number everywhere and 
$$0< \int_X \langle T^n \rangle \le \int_X \langle T \rangle \wedge \eta^{n-1} < \int_X \eta^n.$$
\end{corollary}

\proof Let $\omega$ be a K\"ahler form on $X$.  Let $T$ be the current in Proposition \ref{pro-Tzeropluri}. Thus $T$ has zero Lelong number everywhere and $T$ charges some pluripolar set $A \subset X$. Let $\theta$ be a smooth closed form in the cohomology class of $T$. Let $M>0$ be a big enough constant such that $\eta:=\theta+ M \omega$ is K\"ahler form. Let $T':= T+ M\omega$. Thus $T'$ is cohomologous to $\eta$, and  $T'$ charges $A$ and has zero Lelong number everywhere. By multi-linearity of non-pluripolar products, one has
$$\int_X \langle (T'+ M \omega)^n \rangle \ge \int_X (M\omega)^n >0.$$
The desired inequality 
$$ \int_X \langle T^n \rangle \le \int_X \langle T \rangle \wedge \eta^{n-1}$$
is just a direct consequence of monotonicity of non-pluripolar products (see Theorem \ref{th-mono-current11}). Note that $\langle T \rangle$ is  the non-pluripolar product of only one copy of $T$. It follows that $T= \langle T \rangle + \bold{1}_{E} T$, where $E$ is the polar locus of $T$, that means $E$ is the set where the potentials of $T$ are equal to $-\infty$. Since $\langle T \rangle$ has no mass on pluripolar sets and $T$ has mass on the pluripolar set $A$, we infer that $\bold{1}_E T$ has mass on $A$. In particular $\bold{1}_E T \not =0$. Consequently 
$$\int_X  T \wedge \eta^{n-1} = \int_X \langle T \rangle \wedge \eta^{n-1}+ \int_X \bold{1}_E T \wedge \eta^{n-1}> \int_X \langle T \rangle \wedge \eta^{n-1}.$$
On the other hand, since $T$ is cohomologous to $\eta$, there holds
$$\int_X \eta^n= \int_X  T \wedge \eta^{n-1}.$$
Thus we obtain 
$$\int_X \eta^n > \int_X \langle T \rangle \wedge \eta^{n-1}.$$
This finishes the proof.
\endproof
%Since the current $T$ in Proposition \ref{pro-Tzeropluri} has zero Lelong number everywhere, by Demailly's analytic approximation of psh functions, there exists a sequence of smooth closed positive forms converging to $T$. Hence the cohomology class of $T$ must be nef. 
In the next paragraphs we explain how to use complex Monge-Amp\`ere equations to produce more examples of closed positive $(1,1)$-currents having zero Lelong number everywhere and charging some pluripolar set. 

Let $X$ be now a compact K\"ahler manifold of dimension $n$. Let $\alpha$ be a big cohomology class in $X$ and $\theta$ a smooth closed $(1,1)$-form in $\alpha$. We recall that $\langle \alpha^n \rangle$ is the volume of $\alpha$ which is defined as follows (see \cite{BEGZ}). Let $T_{\min}$ be a closed positive current of minimal singularities in $\alpha$, that means for every closed positive current $T'$ in $\alpha$, one has that $T'$ is more singular than $T_{\min}$, or equivalently, if $u',u_{\min}$   are potentials of $T', T_{\min}$ respectively then $u' \le u_{\min}+ C$ for some constant $C$. Put
 $$V_\theta:= \sup\{u: \text{$u$ is $\theta$-psh and $u \le 0$}\}$$
 which is a $\theta$-psh function. Thus $\ddc V_\theta+ \theta$ is a current of minimal singularity in $\alpha$. We note that currents of minimal singularities are not unique in general (as it is clear in the case where $\alpha$ is K\"ahler). Define 
$$\langle \alpha^n \rangle:= \int_X \langle T_{\min}^n \rangle$$ 
which is independent of the choice of $T_{\min}$ thanks to the monotonicity of non-pluripolar products (see Theorem \ref{th-mono-current11}).

We now assume that  there exists a $\theta$-psh function $u$  such that  $\nu(u,x)=0$ for every $x \in X$, and  
$$0< \int_X \langle (\ddc u+\theta)^n \rangle  < \langle \alpha^n\rangle$$
(the second inequality means that $u$ is not of full Monge-Amp\`ere mass). 
This assumption forces $\alpha$ to be nef because by Demailly's analytic approximation of psh functions, there exists a sequence of smooth closed positive forms converging to $ \ddc u+\theta$, hence,  the cohomology class $\alpha$ of $\ddc u+ \theta$ must be nef. Examples of $u$ and  $\alpha$ are provided by Corollary \ref{cor-currentvanishinlelong}.

%By considering $(u+ V_\theta)/2$ instead of $u$, we can assume moreover that $$\int_X \langle (\ddc u+ \theta)^n \rangle >0.$$
 Let $\mu$ be a (non-zero) non-pluripolar positive  measure on $X$ such that 
$$\mu(X)= \int_{X} \langle (\ddc u +\theta)^n \rangle.$$
Put 
$$P[u]:= \big(\sup\{ \psi \in \PSH(X, \theta): \psi \le \min \{u+C,0\}, \, \text{ for some constant $C$}\}\big)^*,$$
where $\PSH(X, \theta)$ is the set of every $\theta$-psh functions. A crucial property of $P[u]$ is that $P\big[P[u]\big]= P[u]$ (see \cite[Theorem 3.14]{Lu-Darvas-DiNezza-mono}), in other words, $P[u]$ is a model $\theta$-psh function in the sense given in \cite{Lu-Darvas-DiNezza-mono}; see also \cite{Ross-WittNystrom}.

By \cite[Theorem 4.7]{Lu-Darvas-DiNezza-logconcave} (or \cite[Theorem 1.3]{Vu_Do-MA}), there exists a $\theta$-psh function $u_\mu$ such that $ \langle (\ddc u_\mu+\theta)^n \rangle = \mu$, and $u_\mu$ is more singular than $P[u]$. Note that 
$$\int_X \langle (\ddc P[u] + \theta)^n \rangle = \int_X \langle (\ddc u+ \theta)^n\rangle= \int_X \langle (\ddc u_\mu+\theta)^n \rangle$$  
by the monotonicity of non-pluripolar products and the definition of $P[u]$ (e.g., see \cite[Proposition 3.1]{Lu-Darvas-DiNezza-mono}).  

\begin{proposition}\label{prop-lelonggiongnhau} Let $v_1,v_2$ be $\theta$-psh functions on $X$ such that 
\begin{align}\label{eq-fullmass}
0< \int_X \langle (\ddc v_1+ \theta)^n\rangle = \int_X \langle (\ddc v_2+ \theta)^n\rangle,
\end{align}
and $v_1$ is more singular than $P[v_2]$. Then we have $\nu(v_1,x)= \nu(v_2,x)$ for every $x \in X$. 
\end{proposition}

\proof This result was proved in  \cite[Theorem 1.1]{Lu-Darvas-DiNezza-singularitytype} in the case where $v_2= V_\theta$ (hence $P[v_2]= V_\theta$); see also \cite[Theorem 1.1]{Vu_lelong-bigclass}. The proof of \cite[Theorem 1.1]{Lu-Darvas-DiNezza-singularitytype} works without changes for arbitrary $v_2$ with $v_2= P[v_2]$ and $v_1$ is more singular than $v_2$ provided that we have $P[v_1]= P[v_2]$. By  \cite[Theorem 3.14]{Lu-Darvas-DiNezza-mono}, this is always the case because of (\ref{eq-fullmass}). Hence we get 
$$\nu(v_2,x)= \nu(P[v_2],x), \quad \nu(v_1,x)= \nu(P[v_2],x)$$
for every $x$.  The desired assertion thus follows.
\endproof

By Proposition \ref{prop-lelonggiongnhau}, we obtain
$$\nu(u_\mu, x) = \nu(P[u],x)= \nu(u,x)=0$$ 
for every $x \in X$. Clearly for $\mu_1 \not = \mu_2$, we have $u_{\mu_1} \not = u_{\mu_2}$.  This gives us another rich source of $\theta$-psh functions with zero Lelong numbers and  not of full Monge-Amp\`ere mass.

\bibliography{biblio_family_MA,biblio_Viet_papers}
\bibliographystyle{siam}

\bigskip

\noindent
\Addresses
\end{document}